\documentclass[11pt,english]{smfart}

\usepackage[francais,english]{babel}
\usepackage[T1]{fontenc}
\usepackage[latin1]{inputenc}
\usepackage{amsmath,amssymb,amsthm,mathrsfs}
\usepackage[all]{xy}
\usepackage{graphicx}
\usepackage{enumerate}

\numberwithin{equation}{section}

\DeclareMathOperator{\RHom}{RHom}

\DeclareMathOperator{\Hom}{Hom}

\DeclareMathOperator{\HH}{HH}

\DeclareMathOperator{\hh}{hh}
\DeclareMathOperator{\id}{id}
\DeclareMathOperator{\opp}{op}

\DeclareMathOperator{\Hn}{H}
\DeclareMathOperator{\tr}{Tr}

\author{Fran\c{c}ois Petit}
\address{Max Planck Institute for Mathematics\\Vivatsgasse 7\\
53111 Bonn\\
Germany}
\email{petit@mpim-bonn.mpg.de}
\title{A Riemann-Roch Theorem for dg Algebras}

\begin{document}

\theoremstyle{plain} 
\newtheorem{thm}{Theorem}[section]
\newtheorem{cor}[thm]{Corollary}
\newtheorem{prop}[thm]{Proposition}
\newtheorem{lemme}[thm]{Lemma}
\newtheorem{conj}[thm]{Conjecture}
\newtheorem*{theoetoile}{Theorem} 
\newtheorem*{conjetoile}{Conjecture} 
\newtheorem*{theoetoilefr}{Théorème}
\newtheorem*{propetoilefr}{Proposition}

\theoremstyle{definition} 
\newtheorem{defi}[thm]{Definition}
\newtheorem{example}[thm]{Example}
\newtheorem{examples}[thm]{Examples}
\newtheorem{question}[thm]{Question}
\newtheorem{Rem}[thm]{Remark}
\newtheorem{Notation}[thm]{Notation}

\numberwithin{equation}{section}

\newcommand{\On}[1]{\mathcal{O}_{#1}}
\newcommand{\En}[1]{\mathcal{E}_{#1}}
\newcommand{\Fn}[1]{\mathcal{F}_{#1}} 
\newcommand{\tFn}[1]{\mathcal{\tilde{F}}_{#1}}
\newcommand{\hum}[1]{hom_{\mathcal{A}}({#1})}
\newcommand{\hcl}[2]{#1_0 \lbrack #1_1|#1_2|\ldots|#1_{#2} \rbrack}
\newcommand{\hclp}[3]{#1_0 \lbrack #1_1|#1_2|\ldots|#3|\ldots|#1_{#2} \rbrack}
\newcommand{\catMod}{\mathsf{Mod}}
\newcommand{\Der}{\mathsf{D}}
\newcommand{\Ds}{D_{\mathbb{C}}}
\newcommand{\DG}{\mathsf{D}^{b}_{dg,\mathbb{R}-\mathsf{C}}(\mathbb{C}_X)}
\newcommand{\lI}{[\mspace{-1.5 mu} [}
\newcommand{\rI}{] \mspace{-1.5 mu} ]}
\newcommand{\Ku}[2]{\mathfrak{K}_{#1,#2}}
\newcommand{\iKu}[2]{\mathfrak{K^{-1}}_{#1,#2}}
\newcommand{\Be}{B^{e}}
\newcommand{\op}[1]{#1^{\opp}}
\newcommand{\N}{\mathbb{N}}
\newcommand{\Ab}[1]{#1/\lbrack #1 , #1 \rbrack}
\newcommand{\Du}{\mathbb{D}}
\newcommand{\C}{\mathbb{C}}
\newcommand{\Z}{\mathbb{Z}}
\newcommand{\w}{\omega}
\newcommand{\K}{\mathcal{K}}
\newcommand{\Hoc}{\mathcal{H}\mathcal{H}}
\newcommand{\env}[1]{{\vphantom{#1}}^{e}{#1}}
\newcommand{\eA}{{}^eA}
\newcommand{\eB}{{}^eB}
\newcommand{\eC}{{}^eC}
\newcommand{\cA}{\mathcal{A}} 
\newcommand{\cB}{\mathcal{B}}
\newcommand{\cR}{\mathcal{R}}
\newcommand{\cL}{\mathcal{L}}
\newcommand{\cO}{\mathcal{O}}
\newcommand{\cM}{\mathcal{M}}
\newcommand{\cN}{\mathcal{N}}
\newcommand{\cK}{\mathcal{K}}
\newcommand{\cC}{\mathcal{C}}
\newcommand{\Hper}{\Hn^0_{\textrm{per}}}
\newcommand{\Dper}{\Der_{\mathrm{perf}}}
\newcommand{\Yo}{\textrm{Y}}
\newcommand{\gqcoh}{\mathrm{gqcoh}}
\newcommand{\coh}{\mathrm{coh}}
\newcommand{\cc}{\mathrm{cc}}
\newcommand{\qcc}{\mathrm{qcc}}
\newcommand{\qcoh}{\mathrm{qcoh}}
\newcommand{\obplus}[1][i \in I]{\underset{#1}{\overline{\bigoplus}}}
\newcommand{\Lte}{\mathop{\otimes}\limits^{\rm L}}
\newcommand{\pt}{\textnormal{pt}}
\newcommand{\A}[1][X]{\cA_{{#1}}}
\newcommand{\dA}[1][X]{\cC_{X_{#1}}}
\newcommand{\conv}[1][]{\mathop{\circ}\limits_{#1}}
\newcommand{\sconv}[1][]{\mathop{\ast}\limits_{#1}}
\newcommand{\reim}[1]{\textnormal{R}{#1}_!}
\newcommand{\roim}[1]{\textnormal{R}{#1}_\ast}
\newcommand{\ldetens}{\overset{\mathnormal{L}}{\underline{\boxtimes}}}
\newcommand{\br}{\bigr)}
\newcommand{\bl}{\bigl(}
\newcommand{\sC}{\mathscr{C}}
\newcommand{\ucat}{\mathbf{1}}
\newcommand{\ubtimes}{\underline{\boxtimes}}
\newcommand{\uLte}{\mathop{\underline{\otimes}}\limits^{\rm L}} 
\newcommand{\Lp}{\mathrm{L}p}

\frontmatter

\begin{abstract} 
Given a smooth proper dg-algebra $A$, a perfect dg $A$-module $M$ and an
endomorphism $f$ of $M$, we define the Hochschild class of the pair $(M,f)$ with
values in the Hochschild homology of the algebra $A$. Our main result is a Riemann-Roch type
formula involving the convolution of two such Hochschild classes.
\end{abstract}

\maketitle

\section{Introduction}

An algebraic version of the Riemann-Roch formula was recently obtained
by D. Shklyarov \cite{shklyarov} in the framework of the so-called
noncommutative derived algebraic geometry. More precisely, motivated by the well
known result of A. Bondal and M. Van den Bergh about "dg-affinity" of classical
varieties, D. Shklyarov has obtained a formula for the Euler characteristic of the
Hom-complex between two perfect modules over a dg-algebra in terms of the Euler
classes of the modules.

On the other hand, M. Kashiwara and P. Schapira \cite{KS3} initiated an
approach to the Riemann-Roch theorem in the framework of deformation quantization modules
(DQ-modules)  with the view towards applications to various index type theorems. 
Their approach is based on Hochschild homology which,
in this setup, admits a description in terms of the dualizing complexes in
the derived categories of coherent DQ-modules.

The present paper is an attempt to extract some algebraic aspects of this latter approach with the hope that the resulting techniques will provide a uniform point of view for proving Riemann-Roch type results for DQ-modules, D-modules etc. (e.g. the Riemann-Roch-Hirzebruch formula for traces of differential operators obtained by M. Engeli and G. Felder \cite{Felder}). 
 In this paper, we obtain a Riemann-Roch theorem in the dg setting, similarly as D. Shklyarov.
However, our approach is really different of the latter one in that we avoid the categorical definition of the Hochschild homology, and use instead  the Hochschild homology of the ring $A$ expressed in terms of dualizing objects. Our result is slightly more general than the one obtained in \cite{shklyarov}. Instead of a kind of non-commutative Riemann-Roch theorem, we rather prove a kind of non-commutative Lefschetz theorem. Indeed, it involves certain Hochschild classes of \textit{pairs} $(M,f)$ where $M$ is a perfect dg module over a smooth proper dg
algebra and $f$ is an endomorphism of $M$ in the derived category of perfect
$A$-modules. Moreover, our approach follows \cite{KS3}. In particular, we have in our setting relative finiteness and duality results (Theorem \ref{tensor} and Theorem \ref{contraction})  that may be compared with \cite[Theorem 3.2.1]{KS3} and \cite[Theorem 3.3.3]{KS3}. Notice that the idea to approach the classical Riemann-Roch theorem for smooth projective varieties via their Hochschild homology goes back at least to the work of N. Markarian \cite{Markarian}. This approach was developed further by A.
C\u{a}ld\u{a}raru \cite{CaldararuI}, \cite{Caldararu1/2} and A. C\u{a}ld\u{a}raru, S. Willerton
\cite{caldararu} where, in particular, certain purely categorical aspects of the
story were emphasized. The results of \cite{CaldararuI} suggested
that a Riemann-Roch type formula might exist for triangulated categories of quite
general nature, provided they possess Serre duality. In this categorical framework,
the role of the Hochschild homology is played by the space of morphisms from the
inverse of the Serre functor to the identity endofunctor. 
In a sense, our result can be viewed as a non-commutative generalization of A.
C\u{a}ld\u{a}raru's version of the topological Cardy condition \cite{CaldararuI}. Our
original motivation was different though it also came from the theory of
DQ-modules \cite{KS3}.\\ 

Here is our main result:

\begin{theoetoile}
Let $A$ be a proper, homologically smooth dg algebra, $M \in \Der_{\mathrm{perf}}(A), \; f \in \Hom_A(M,M)$ and $ N \in
\Der_{\mathrm{perf}}(\op{A}), \; g \in \Hom_{\op{A}}(N,N)$.

Then

\begin{equation*}
\hh_k(N \Lte_{A} M, g \Lte_{A} f)=  \hh_{\op{A}}(N,g) \cup \hh_A(M,f),
\end{equation*}
where $\cup$ is a pairing between the corresponding Hochschild
homology groups and where $\hh_A(M,f)$ is the Hochschild class of the pair $(M,f)$ with 
value in the Hochschild homology of $A$.
\end{theoetoile}

The above pairing is obtained using Serre duality in the derived category of perfect complexes 
and, thus, it strongly resembles analogous pairings, studied in
some of the references previously mentioned (cf. \cite{CaldararuI}, \cite{KS3},
\cite{Shklyarov1}). We prove that various methods of constructing a pairing on Hochschild homology lead to the same result. Notice that in  \cite{Ramadoss_Pair},  A. Ramadoss studied the links between different pairing on Hochschild homology.

To conclude, we would like to mention the recent paper by A.
Polishchuk and A. Vaintrob \cite{Polishchuk} where a categorical version of the
Riemann-Roch theorem was applied in the setting of the so-called Landau-Ginzburg
models (the categories of matrix factorizations). We hope that our results, in
combination with some results by D. Murfet \cite{Murfet}, may provide an alternative way
to derive the Riemann-Roch formula for singularities.\\

\noindent
\textbf{Acknowledgement.} I would like to express my gratitude to my advisor P. Schapira for suggesting this problem to me, for his guidance and for his support. I am very grateful to Y. Soibelman for carefully explaining his article \cite{KoSY} to me and to B. Keller for patiently answering my questions about dg categories. I would like to thank A. C\u{a}ld\u{a}raru for explaining his articles to me and K. Ponto for explanations concerning her works \cite{ponto}. I would like to warmly thank D. Shklyarov and G. Ginot for many useful discussions and precious advices. Finally, I would like to thank Damien Calaque and Michel Vaquié for their thorough reading of the manuscript and numerous remarks which have allowed substantial improvements.

\section{Conventions}\label{notations}

All along this paper $k$ is a field of characteristic zero. A $k$-algebra is a $k$-module $A$ equiped  with an associative $k$-linear multiplication admitting a two sided unit $1_A$.  

All the graded modules considered in this paper are cohomologically $\Z$-graded. We abbreviate differential graded module (resp. algebra) by dg module (resp. dg algebra). 

If $A$ is a dg algebra and $M$ and $N$ are dg $A$-modules, we write $\Hom_A^\bullet(M,N)$ for the total $\Hom$-complex.  

If $M$ is a dg $k$-module, we define $M^\ast=\Hom_k^\bullet(M,k)$ where $k$ is considered as the dg $k$-module whose 0-th components is $k$ and other components are zero.

We write $\otimes$ for the tensor product over $k$. If $x$ is an homogeneous element of a differential graded module we denote by $|x|$ its degree. 

If $A$ is a  dg algebra we will denote by $\op{A}$ the opposite dg algebra. It is the same as a differential graded $k$-module but the multiplication is given by $a \cdot b= (-1)^{|a||b|}b a$. We denote by $A^e$ the dg algebra $A \otimes \op{A}$ and by $\eA$ the algebra $\op{A} \otimes A$.\\

By a module we understand a left module unless otherwise specified. If $A$ and $B$ are dg algebras, $A$-$B$ bimodules will be considered as left $A \otimes \op{B}$-modules via the action 
\begin{equation*}
a \otimes b \cdot m = (-1)^{|b||m|} amb.
\end{equation*}
If we want to emphasize the left (resp. right) structure of an $A$-$B$ bimodule we will write $_AM$ (resp. $M_B$). If $M$ is an $A \otimes \op{B}$-modules, then we write $\op{M}$ for the corresponding $\op{B} \otimes A$-module. Notice that $(\op{M})^\ast \simeq \op{(M^\ast)}$ as $B\otimes \op{A}$-modules.\\

\section{Perfect modules}\label{Serre}

\subsection{Compact objects}
We recall some classical facts concerning compact objects in triangulated categories. We refer the reader to \cite{Nee_book}.

Let $\mathcal{T}$ be a triangulated category admitting arbitrary small coproducts.

\begin{defi}
 An object $M$ of $\mathcal{T}$ is compact if for each family $(M_i)_{i \in I}$ of objects of $\mathcal{T}$ the canonical morphism 
\begin{equation}
\underset{i \in I}{\bigoplus}\Hom_{\mathcal{T}}(M,M_i) \to \Hom_\mathcal{T}(M,\underset{i \in I}{\bigoplus} M_i)
\end{equation}
is an isomorphism. We denote by $\mathcal{T}^{c}$ the full subcategory of $\mathcal{T}$ whose objects are the compact objects of $\mathcal{T}$.
\end{defi}

Recall that a triangulated subcategory of $\mathcal{T}$ is called thick if it is closed under isomorphisms and direct summands.

\begin{prop}
The category $\mathcal{T}^c$ is a thick subcategory of $\mathcal{T}$.
\end{prop}

We prove the following fact that will be of constant use.

\begin{prop} \label{thickiso}
Let $\mathcal{T}$ and $\mathcal{S}$ be two triangulated categories and $F$ and $G$ two functors of triangulated categories from $\mathcal{T}$ to $\mathcal{S}$ with a natural transformation $\alpha:F \Rightarrow G$. Then the full subcategory $\mathcal{T}_\alpha$ of $\mathcal{T}$ whose objects are the $X$ such that $\alpha_X:F(X) \to G(X)$ is an isomorphism, is a thick subcategory of $\mathcal{T}$.
\end{prop}

\begin{proof}
The category $\mathcal{T}_\alpha$ is non-empty since $0$ belongs to it and it is stable by shift since $F$ and $G$ are functors of triangulated categories. Moreover, the category $\mathcal{T}_\alpha$ is a full subcategory of a triangulated category. Thus, to verify that $\mathcal{T}_\alpha$ is triangulated, it only remains to check that it is stable by taking cones. Let $f: X \to Y$ be a morphism of $\mathcal{T}_\alpha$. Consider a distinguished triangle in $\mathcal{T}$
\begin{equation*}
X \stackrel{f}{\to} Y \to Z \to X[1].
\end{equation*}
We have the following diagram
\begin{equation*}
\xymatrix{ F(X) \ar[r]^-{F(f)} \ar[d]^-{\wr}_-{\alpha_X} & F(Y) \ar[r] \ar[d]^-{\wr}_-{\alpha_Y} & F(Z) \ar[r] \ar[d]_-{\alpha_Z} & F(X[1])  \ar[d]^-{\wr}_-{\alpha_X[1]}\\
G(X) \ar[r]^-{G(f)}  & G(Y) \ar[r]  & G(Z) \ar[r]  & G(X[1]).   
}
\end{equation*}
By the five Lemma it follows that the morphism $\alpha_Z$ is an isomorphism. Therefore, $Z$ belongs to $\mathcal{T}_\alpha$. This implies that the category $\mathcal{T}_\alpha$ is triangulated.

It is clear that $\mathcal{T}_\alpha$ is closed under isomorphism. Since $F$ and $G$ are functors of triangulated categories they commutes with finite direct sums. Thus, $\mathcal{T}_\alpha$ is stable under taking direct summands. It follows that $\mathcal{T}_\alpha$ is a thick subcategory of $\mathcal{T}$.
\end{proof}

\begin{defi}\label{defcomge}
The triangulated category $\mathcal{T}$ is compactly generated if there is a set $\mathcal{G}$ of compact objects $G$ such that an object $M$ of $\mathcal{T}$ vanishes if and only if we have $\Hom_{\mathcal{T}}(G[n],M) \simeq 0$ for every $G \in \mathcal{G}$ and $n \in \Z$.
\end{defi}

\begin{thm}[\cite{Nee_comp, Rav}]\label{description}
Let $\mathcal{G}$ be as in Definition \ref{defcomge}. An object of $\mathcal{T}$ is compact if and only if it is isomorphic to a direct factor of an iterated triangle extension of copies of object of $\mathcal{G}$ shifted in both directions.
\end{thm}

\begin{Rem}\label{remthick}
The above theorem implies that the category of compact objects is the smallest thick subcategory of $\mathcal{T}$ containing $\mathcal{G}$.
\end{Rem}

\subsection{The category of perfect modules}

In this section, following \cite{Keller_dg}, we recall the definition of the category of perfect modules.

Let $A$ be a differential graded algebra. One associates to $A$ its category of differential graded modules, denoted $\mathcal{C}(A)$, whose objects are the differential graded modules and whose morphisms are the morphisms of chain complexes.

Recall that the category $\mathcal{C}(A)$ has a compactly generated model structure, called the projective structure, where the weak equivalences are the quasi-isomorphisms, the fibrations are the level-wise surjections. The reader may refer to \cite{Hovey} for model categories  and to \cite[ch.11]{Fresse} for a detailed account on the projective model structure of $\mathcal{C}(A)$.

The derived category $\Der(A)$ is the localisation of $\mathcal{C}(A)$ with respects to the class of quasi-isomorphisms. The category $\Der(A)$ is a triangulated category, it admits arbitrary coproducts and is compactly generated by the object $A$.  
Theorem \ref{description} leads to Proposition \ref{descriptionAmod} which allows us to define perfect modules in terms of compact objects.

\begin{prop}\label{descriptionAmod}
An object of $\Der(A)$ is compact if and only if it is isomorphic to a direct factor of an iterated extension of copies of $A$ shifted in both directions.
\end{prop}

\begin{defi}
A differential graded module is perfect if it is a compact object of $\Der(A)$. We write $\Dper(A)$ for the category of compact objects of $\Der(A)$.
\end{defi}

\begin{Rem}
This definition implies immediatly that $M$ is a perfect $k$-module if and only if $\sum_i \dim_k \Hn^i(M) < \infty$.
\end{Rem}

A direct consequence of Proposition \ref{descriptionAmod} is

\begin{prop}\label{cons}
Let A and B be two dg algebras and $F:\Der(A) \rightarrow \Der(B)$ a functor of triangulated categories. Assume that  $F(A)$ belongs to $\Dper(B)$. Then, for any $X$ in $\Dper(A)$ ,  $F(X)$ is an object of $\Dper(B)$.
\end{prop}
%
%

\begin{prop} \label{isoext}
Let A and B be two dg algebras and $F,G:\Der(A) \rightarrow \Der(B)$ two functors of triangulated categories and $\alpha:F \Rightarrow G$ a natural transformation. If $\alpha_A:F(A) \to G(A)$ is an isomorphism then $\alpha_M:F(M) \to G(M)$ is an isomorphism for every $M \in \Dper(A)$.
\end{prop}

\begin{proof}
By Proposition \ref{thickiso}, the category $\mathcal{T}_\alpha$ is thick. This category contains $A$ by hypothesis. It follows by Remark \ref{remthick} that $\mathcal{T}_\alpha=\Dper(A)$.
\end{proof}

\begin{defi}
A dg $A$-module $M$ is a finitely generated semi-free module if it can be obtain by iterated extensions of copies of $A$ shifted in both directions. 
\end{defi}

\begin{prop}
\begin{enumerate}[(i)]
\item  Finitely generated semi-free modules are cofibrant objects of $\mathcal{C}(A)$ endowed with the projective structure.  
\item A perfect module is quasi-isomorphic to a direct factor of a finitely generated semi-free module.
\end{enumerate}
\end{prop}

\begin{proof}
\begin{enumerate}[(i)]
\item is a direct consequence of \cite[Proposition 11.2.9]{Fresse}.

\item follows from Proposition \ref{descriptionAmod} and from the facts that, in the projective structure, every object is fibrant and finitely generated semi-free modules are cofibrant.%
\end{enumerate}%
\end{proof}

\begin{Rem}
The above statement is a special case of \cite[Proposition 2.2]{Toen}.
\end{Rem}

\subsection{Finiteness results for perfect modules}

We summarize some finiteness results for perfect modules over a dg algebra satisfying suitable finiteness and regularity hypothesis. The main reference for this section is \cite{Toen}. 
Most of the statements of this subsection and their proofs can be found in greater generality in \cite[§2.2]{Toen}. For the sake of completeness, we give the proof of these results in our specific framework.

\begin{defi} 
A dg $k$-algebra $A$ is said to be proper if it is perfect over $k$. 
\end{defi}


The next theorem, though the proof is much easier, can be thought as a dg analog to the theorem asserting the finiteness of proper direct images for coherent $\mathcal{O}_X$-modules.

\begin{thm}\label{tensor}
Let $A$, $B$ and $C$ be dg algebras. Assume $B$ is a proper dg algebra. Then the functor $\cdot \Lte_B \cdot : \Der(A \otimes \op{B}) \times \Der(B \otimes \op{C}) \to \Der(A \otimes \op{C})$  induces a functor $\cdot \Lte_B \cdot : \Dper(A \otimes \op{B}) \times \Dper(B \otimes \op{C}) \to \Dper(A \otimes \op{C})$.
\end{thm}

\begin{proof}
According to Proposition \ref{cons}, we only need to check that $(A \otimes \op{B}) \Lte_B (B \otimes \op{C}) \simeq A \otimes \op{B} \otimes \op{C} \in \Dper(A \otimes \op{C})$. In $\mathcal{C}(k)$, $B$ is homotopically equivalent to $\Hn(B):=\bigoplus_{n \in \Z} \Hn^n(B)[n]$ since $k$ is a field. Then, in $\mathcal{C}(A \otimes \op{C})$, $ A \otimes \op{B} \otimes \op{C}$ is homotopically equivalent to $A \otimes \Hn(\op{B}) \otimes \op{C}$ which is a finitely generated free $A\otimes \op{C}$-module since $B$ is proper.
\end{proof}

We recall a regularity condition for dg algebra called homological smoothness, \cite{KoSY}, \cite{Toen}.

\begin{defi}
A dg-algebra $A$ is said to be homologically smooth if $A \in \Der_{\mathrm{perf}}(A^e)$.
\end{defi}

\begin{prop}
The tensor product of two homologically smooth dg-algebras is an homologically smooth dg-algebra. 
\end{prop}

\begin{proof}
Obvious.
\end{proof}

There is the following characterization of perfect modules over a proper homologically smooth dg algebra extracted from \cite[Corollary 2.9]{Toen}.

\begin{thm} \label{dual vect}
Let A be a proper dg algebra. Let $N \in \Der(A)$. 
\begin{enumerate}[(i)]
\item If $N \in \Der_{\mathrm{perf}}(A)$ then  $N$ is perfect over $k$. 
\item If $A$ is homologically smooth and  N is perfect over $k$ then $N \in \Der_{\mathrm{perf}}(A)$. 
\end{enumerate}
\end{thm}

\begin{proof}
We follow the proof of \cite{Shklyarov1}.

\textit{(i)} Apply Proposition \ref{cons}.\\

\textit{(ii)} Assume that $N \in \Der(A)$ is perfect over $k$. 
Let us show that the triangulated functor $\cdot \Lte_A N:\Der(A^e) \to \Der(A)$ induces a triangulated functor $\cdot \Lte_A N:\Dper(A^e) \to \Dper(A)$. Let $pN$ be a cofibrant replacement of $N$. Then
\begin{equation*}
A^e \Lte_A N \simeq A^e \otimes_A pN \simeq A \otimes pN.
\end{equation*}%
In $\mathcal{C}(k)$, $pN$ is homotopically equivalent to $\Hn(pN):=\bigoplus_{n \in \Z} \Hn^n(pN)[n]$. Thus, there is an isomorphism in $\Der(A)$ between $A \otimes_k pN$ and $A \otimes_k \Hn(pN)$. The dg A-module $A \otimes_k \Hn(pN)$ is perfect. Thus, by Proposition \ref{cons}, the functor $\cdot \Lte_A N$ preserves perfect modules. Since $A$ is homologically smooth, $A$ belongs to $\Dper(A^e)$. Then, $A \Lte_A N \simeq N$ belongs to $\Dper(A)$.  
%
\end{proof}

A similar argument leads to (see \cite[Lemma 2.6]{Toen})

\begin{lemme}\label{Ext finite}
If A is a proper algebra then $\Dper(A)$ is Ext-finite.
\end{lemme}

\subsection{Serre duality for perfect modules}
In this subsection, we recall some facts concerning Serre duality for perfect modules over a dg algebra and give various forms of the Serre functor in this context. References are made to \cite{BKSerre}, \cite{Gin}, \cite{shklyarov}.

Let us recall the definition of a Serre functor, \cite{BKSerre}.
\begin{defi}
Let $\mathcal{C}$ be a $k$-linear  Ext-finite triangulated category. A Serre functor $S:\mathcal{C} \rightarrow \mathcal{C}$ is an autoequivalence of $\mathcal{C}$ such that there exist an isomorphism
\begin{equation}
\Hom_{\mathcal{C}}(Y,X)^{\ast} \simeq \Hom_{\mathcal{C}}(X,S(Y))
\end{equation} 
functorial with respect to $X$ and $Y$ where $\ast$ denote the dual with respect to $k$.
If it exists, such a functor is unique up to natural isomorphism.
\end{defi}

\begin{Notation}
We set $\mathbb{D}^\prime_A=\RHom_A(\cdot,A): \op{(\Der(A))} \to \Der(\op{A})$.
\end{Notation}

\begin{prop}\label{dual equi}
The functor $\Du^{\prime}_A$ preserves perfect modules and induces an equivalence $\op{(\Der_{\mathrm{perf}}(A))} \rightarrow \Der_{\mathrm{perf}}(\op{A})$. When restricted to perfect modules, $\Du^{\prime}_{\op{A}} \circ \Du^{\prime}_A \simeq \id$.
\end{prop}

\begin{proof}
See \cite[Proposition A.1]{Shklyarov1}.
\end{proof}

\begin{prop}\label{dualhom}
Suppose N is a perfect A-module and M is an arbitrary left $A \otimes \op{B}$-module, where B is another dg algebra. Then there is a natural isomorphism of $B$-modules
\begin{equation}\label{forumule dualhom}
(\RHom_A(N,M))^\ast \simeq \op{(M^\ast)} \Lte_A N 
\end{equation}
\end{prop} 


\begin{thm} \label{isoSerre}
In $\Dper(A)$, the endofunctor $\RHom_A(\cdot,A)^{\ast}$ is isomorphic to the endofunctor $(\op{A})^\ast \Lte_{A} -$.       
\end{thm}

\begin{proof}
This result is a direct corollary of Proposition \ref{dualhom} by choosing $M=A$ and $B=A$.
\end{proof}

Lemma \ref{intercontra} and Theorem \ref{contraction} are probably well known results. Since we do not know any references for them, we shall give detailed proofs.

\begin{lemme}\label{intercontra}
Let B be a proper dg algebra, $M \in \Der_{\mathrm{perf}}(A \otimes \op{B})$ and $N \in \Der_{\mathrm{perf}}(B \otimes \op{C})$. There are the following canonical isomorphisms respectively in $\Dper(\op{B} \otimes C)$ and $\Dper(\op{A} \otimes C)$:

\begin{equation}\label{dual anneau}
B^{\ast} \Lte_{\op{B}} \RHom_{B \otimes \op{C}}(N,B \otimes \op{C})  \simeq \RHom_{\op{C}}(N,\op{C})
\end{equation}

\begin{align}\label{calcul inter}
\hspace{1.25cm}\RHom_{A \otimes \op{B}}(M,A\otimes  \RHom_{\op{C}}(N, \op{C}))&  \notag \\
  &\hspace{-2.5cm}\simeq\RHom_{A \otimes \op{C}}(M \Lte_{B} N,A \otimes \op{C}).
\end{align}
\end{lemme}

\begin{proof}

\textit{(i)} Let us prove formula (\ref{dual anneau}). Let $N \in \mathcal{C}(B \otimes \op{C})$. There is a morphism of $\op{B} \otimes C$ modules
\begin{equation*}
\Psi_N: B^{\ast} \otimes_{\op{B}} \Hom_{B \otimes \op{C}}^\bullet(N, B \otimes \op{C}) \to \Hom_{\op{C}}^\bullet(N, \op{C})
\end{equation*} 
such that $\Psi_N(\delta \otimes_{\op{B}} \phi)= m \circ (\delta \otimes \id_{\op{C}}) \circ \phi$ where $m: k \otimes \op{C} \to \op{C}$ and $m(\lambda \otimes c)=\lambda \cdot c$. Clearly, $\Psi$ is a natural transformation between the functor $ B^{\ast} \otimes_{\op{B}} \Hom_{B \otimes \op{C}}^\bullet(\cdot, B \otimes \op{C})$ and $ \Hom_{\op{C}}^\bullet(\cdot, \op{C})$. For short, we set
\begin{align*}
F(X)= B^{\ast} \Lte_{\op{B}} \RHom_{B \otimes \op{C}}(X,B \otimes \op{C})& \hspace{0.5cm} \mathrm{and}  & G(X)=\RHom_{\op{C}}(X,\op{C}).
\end{align*}
If $X$ is a direct factor of a finitely generated semi-free $B \otimes \op{C}$-module, we obtain that  $\RHom_{B \otimes \op{C}}(X,B \otimes \op{C}) \simeq \Hom_{B \otimes \op{C}}^\bullet(X,B \otimes \op{C})$ and the $\op{B} \otimes C$-module $\Hom_{B \otimes \op{C}}^\bullet(X,B \otimes \op{C})$ is flat over $\op{B}$ since it is flat over $\op{B} \otimes C$. By Lemma 3.4.2 of \cite{Hinch} we can use flat replacements instead of cofibrant one to compute derived tensor products. Thus $F(X)\simeq B^{\ast} \otimes_{\op{B}} \Hom_{B \otimes \op{C}}(X,B \otimes \op{C})$. 

Since $B \otimes \op{C}$ is a cofibrant $\op{C}$-module, it follows that the forgetful functor from $\cC(B \otimes \op{C})$ to $\cC(\op{C})$ preserves cofibrations. Thus, X is a cofibrant $\op{C}$- module. It follows that $G(X)\simeq\Hom_{\op{C}}(X,\op{C})$. Therefore, $\Psi$ induces a natural transformation from $F$ to $G$ when they are restricted to $\Dper(B \otimes \op{C})$.

Assume that $X=B \otimes \op{C}$. Then we have the following commutative diagram
\begin{equation*}
\xymatrix{B^{\ast} \otimes_{\op{B}} \Hom_{B \otimes \op{C}}^\bullet(B \otimes \op{C},B \otimes \op{C})  \ar[r]^-{\Psi_{B \otimes \op{C}}} \ar[d]^-{\wr} & \Hom_{\op{C}}^\bullet(B \otimes \op{C}, \op{C}) \ar[d]^-{\wr}\\
B^\ast \otimes_{\op{B}} \op{B} \otimes C  \ar[d]^-{\wr}& \Hom_k^\bullet(B, \Hom_{\op{C}}^\bullet(\op{C},\op{C})\ar[d]^-{\wr}\\
B^{\ast} \otimes C \ar[r]^-{\sim}& \Hom_k^\bullet(B,C)}
\end{equation*}
which proves that $\Psi_{B \otimes \op{C}}$ is an isomorphism. The bottom map of the diagram is an isomorphism because $B$ is proper. Hence, by Proposition \ref{isoext}, $\Psi_X$ is an isomorphism for any $X$ in $\Dper(B \otimes \op{C})$ which proves the claim.

\textit{(ii)} Let us prove formula (\ref{calcul inter}). We first notice that there is a morphism of $A \otimes \op{C}$-modules functorial in $M$ and $N$
\begin{equation*}
\Theta:\Hom_{A \otimes \op{B}}^\bullet(M, A \otimes \Hom_{\op{C}}^\bullet(N, \op{C})) \to \Hom_{A \otimes \op{C}}^\bullet(M \otimes_{B} N,A \otimes \op{C})
\end{equation*}
defined by $\psi \mapsto (\Psi: m\otimes n \mapsto \psi(m)(n))$.\\
If $M=A \otimes \op{B}$ and $N= B \otimes \op{C}$, then it induces an isomorphism. By applying an argument similar to the previous one, we are able to establish the isomorphism (\ref{calcul inter}).
\end{proof}

The next relative duality theorem can be compared to \cite[Thm 3.3.3]{KS3} in the framework of DQ-modules though the proof is completely different.

\begin{thm}\label{contraction}
Assume that B is proper. Let $M \in \Der_{\mathrm{perf}}(A \otimes \op{B})$ and $N \in \Der_{\mathrm{perf}}(B \otimes \op{C})$. There is a natural isomorphism in $\Der_{\mathrm{perf}}(\op{A} \otimes C)$
\begin{equation*}
\Du^{\prime}_{A \otimes \op{B}} (M) \Lte_{\op{B}} B^{\ast} \Lte_{\op{B}}  \Du^{\prime}_{B \otimes \op{C}}(N) \simeq \Du^{\prime}_{A \otimes \op{C}}( M \Lte_{B} N).
\end{equation*}
\end{thm}

\begin{proof}
We have
\begin{align*}
\Du^{\prime}(M) \Lte_{B^{\opp}} B^\ast \Lte_{B^{\opp}}  \Du^{\prime}(N)& \\
&\hspace{-2.5cm} \simeq\RHom_{A \otimes \op{B}}(M,A\otimes B^{\opp}) \Lte_{B^{\opp}} B^\ast \Lte_{B^{\opp}} \RHom_{B \otimes \op{C}}(N, B \otimes \op{C})\\
                                                 &\hspace{-2.5cm}\simeq \RHom_{A \otimes \op{B}} (M,A \otimes \op{B}) \Lte_{\op{B}} \RHom_{\op{C}}(N, \op{C})\\
                                                 &\hspace{-2.5cm}\simeq \RHom_{A \otimes \op{B}}(M,A\otimes  \RHom_{\op{C}}(N, \op{C}))\\
                                                  &\hspace{-2.5cm}\simeq \RHom_{A \otimes \op{C}}(M \Lte_{B} N,A \otimes \op{C}).
\end{align*}
\end{proof}

One has (see for instance \cite{Gin})

\begin{thm}
Let A be a proper homologically smooth dg algebra. The functor  $N \mapsto (\op{A})^{\ast} \Lte_A N, \Dper(A) \to \Dper(A)$ is a Serre functor.
\end{thm}

\begin{proof}
According to Lemma \ref{Ext finite}, $\Der_{\mathrm{perf}}(A)$ is an Ext-finite category. By Theorem \ref{isoSerre}, the functor $(\op{A})^\ast \Lte_A -$ is isomorphic to the functor $\RHom_A(\cdot,A)^\ast$. Moreover using Theorem \ref{dual vect} and Proposition \ref{dual equi} one sees that $\RHom_A(\cdot,A)^{\ast}$ is an equivalence on $\Dper(A)$ and so is the functor $(\op{A})^{\ast} \Lte_A \cdot$.
By applying Theorem \ref{contraction} with $A=C=k$, $B=\op{A}$, $N=M$ and $M=\RHom_A(N,A)$ one obtains
\begin{equation*}
\RHom_A(N, (\op{A})^{\ast} \Lte_A M) \simeq \RHom_A(M,N)^\ast.
\end{equation*} %
\end{proof}%

\begin{defi}
One sets $S_A:\Dper(A) \to \Dper(A)$, $N \mapsto (\op{A})^\ast \Lte_A N$ for the Serre functor of $\Dper(A)$. 
\end{defi}

The Serre functor can also be expressed in term of dualizing objects. They are defined by \cite{KoSY}, \cite{VdBduality}, \cite{Gin}. Related results can also be found in \cite{jorg}. One sets:
\begin{equation*} 
 \omega_{A}^{-1}:=\RHom_{^{e} \! A}(\op{A}, \eA)=\Du^{\prime}_{\eA}(\op{A}),
\end{equation*} 
\begin{equation*} 
 \w_A:=\RHom_{A}(\w_A^{-1},A)= \Du_A^\prime(\w^{-1}_A).
\end{equation*}
The structure of $A^e$-module of $\w_A^{-1}$ is clear. The object $\w_A$ inherits a structure of $\op{A}$-module from the structure of $\op{A}$-module of $A$ and a structure of $A$-module from the structure of $\op{A}$-module of $\w_A^{-1}$. This endows $\w_A$ with a structure of $A^e$-modules.

Since $A$ is a smooth dg algebra, it is a perfect $A^e$-module. Proposition \ref{dual equi} ensures that $\w_A^{-1}$ is a perfect $A^e$-module. Finally, Proposition \ref{dual vect} shows that $\w_A$ is a perfect $A^e$-module. 

\begin{prop}
The functor $\omega_A^{-1} \Lte_{A} -$ is left adjoint to the functor $\w_A \Lte_{A} -$.
\end{prop}%

One also has, \cite{Gin}
\begin{thm}
The two functors $\omega_A^{-1} \Lte_A -$ and $S_A$ from $\Der_{\mathrm{perf}}(A)$ to $\Der_{\mathrm{perf}}(A)$  are quasi-inverse.
\end{thm}

\begin{proof}
The functor $S_A$ is an autoequivalence. Thus, it is a right adjoint of its inverse. We prove that $\omega_A^{-1} \Lte_A -$ is a left adjoint to $S_A$. On the one hand we have for every $N,M \in \Der_{\mathrm{perf}}(A)$ the isomorphism
\begin{equation*}
\Hom_{\Der_{\mathrm{perf}}(A)}(N,S_A(M)) \simeq (\Hom_{\Der_{\mathrm{perf}}(A)}(M,N))^\ast.
\end{equation*}  
On the other hand we have the following natural isomorphisms
\begin{equation*}
\begin{split}
\RHom_A(\omega_A^{-1} \Lte_A N ,M) \simeq & (M^\ast \Lte_A \omega_A^{-1} \Lte_{A} N)^\ast\\
                                 \simeq & ((\w_A^{-1})^{\opp} \Lte_{A^e} (N \otimes M^\ast)  )^\ast\\
                                 \simeq & \RHom_{A^e}(\RHom_{\eA}(\omega_{\op{A}}^{-1},\eA), (N \otimes M^\ast))^\ast\\
                                 \simeq & \RHom_{A^e}(A,(N \otimes M^\ast ))^\ast.\\                               
\end{split}
\end{equation*}
Using the isomorphism
\begin{equation*}
\Hom_{\Der_{\mathrm{perf}}(A^{e})}(A, N \otimes M^\ast) \simeq \Hom_{\Der_{\mathrm{perf}}(A)}(M,N),
\end{equation*}
we obtain the desired result.
\end{proof}

\begin{cor}
The functors $\w_A^{-1}\Lte_A \cdot : \Dper(A) \to \Dper(A)$ and $\w_A \Lte_A \cdot : \Dper(A) \to \Dper(A)$ are equivalences of categories.
\end{cor}

\begin{cor}
The natural morphisms in $\Dper(A^e)$
\begin{align*}
A \to & \RHom_{A}(\w_A^{-1},\w_A^{-1})\\
A  \to & \RHom_{A}(\w_A,\w_A)
\end{align*}
are isomorphisms.
\end{cor}

\begin{proof}
The functor $\w_A^{-1}\Lte_A \cdot$ induces a morphism in $\Der(A^e)$ 

\begin{equation*}
A\simeq\RHom_A(A,A) \to \RHom_{A}(\w_A^{-1}\Lte_A A,\w_A^{-1} \Lte_A A)\simeq \RHom_{A}(\w_A^{-1},\w_A^{-1}).
\end{equation*}

Since $\w_A^{-1}\Lte_A \cdot$ is an equivalence of category, for every $i \in \Z$
\begin{equation*}
\Hom_A(A,A[i]) \stackrel{\sim}{\to} \Hom_{A}(\w_A^{-1},\w_A^{-1}[i]).
\end{equation*}
The results follows immediately. The proof is similar for the second morphism.
\end{proof}

\begin{prop}\label{invinitial}
Let A be a proper homologically smooth dg algebra. We have the isomorphisms of $A^e$-modules
\begin{equation*}
\w_A \Lte_A \omega_A^{-1} \simeq A, \qquad \omega_A^{-1} \Lte_A \w_A \simeq A.
\end{equation*}
\end{prop}

\begin{proof}
We have
\begin{align*}
\w_A \Lte_A \w_A^{-1} &\simeq \RHom_A(\w_A^{-1},A) \Lte_A \w_A^{-1}\\
&\simeq \RHom_A(\w_A^{-1},\w_A^{-1})\\
&\simeq A.
\end{align*}
For the second isomorphism, we remark that

\begin{align*}
\RHom_A(\w_A,A) \Lte_A \w_A &\simeq \RHom_A(\w_A,\w_A)\\
&\simeq A,
\end{align*}%
and
\begin{align*}
\RHom_A(\w_A,A) &\simeq \RHom_A(\w_A,A) \Lte_A \w_A \Lte_A \w_A^{-1}\\ 
&\simeq \w_A^{-1}\\
\end{align*}%
which conclude the proof.
\end{proof} 

\begin{cor}
Let A be a proper homologically smooth dg algebra. The two objects $(\op{A})^\ast$ and $\omega_A$ of $\Der_{\mathrm{perf}}(A^e)$ are isomorphic. 
\end{cor}

\begin{proof}
Applying Theorem \ref{contraction} with $A=B=C=\op{A}$, $M=N=\op{A}$, we get that $\w_A^{-1} \Lte_A (\op{A})^\ast \Lte_A \w_A^{-1} \simeq \w_A^{-1}$ in $\Dper(A^e)$. Then, the result follows from Corollary \ref{invinitial}.
\end{proof}

\begin{Rem}
Since $\w_A$ and $(\op{A})^\ast$ are isomorphic as $A^e$-modules, we will use $\w_A$ to denote both $(\op{A})^\ast$ and $\RHom_A(\w_A^{-1},A)$  considered as the dualizing complexes of the category $\Der_{\mathrm{perf}}(A)$.
\end{Rem}

The previous results allow us to build an "integration" morphism. 

\begin{prop}\label{integration}
There exists a natural "integration" morphism in $\Der_{\mathrm{perf}}(k)$
\begin{equation*}
\w_{\op{A}} \Lte_{A^e} A  \rightarrow k.
\end{equation*}
\end{prop}

\begin{proof}
There is a natural morphism $k \rightarrow \RHom_{A^e}(A,A)$. Applying $(\cdot)^{\ast}$ and formula (\ref{forumule dualhom}) with $A=A^e$ and $B=k$, we obtain a morphism $A^\ast \Lte_{A^e} A  \rightarrow k$. Here, $A^\ast$ is endowed with its standard structure of right $A^e$-modules that is to say with its standard structure of left $\eA$-module. Thus, $\w_{\op{A}} \Lte_{A^e} A \simeq A^\ast \Lte_{A^e} A  \rightarrow k$.
%
\end{proof}

\begin{cor}
There exists a canonical map $\w_A \to k$ in $\Dper(k)$ induced by the morphism of Proposition \ref{integration}.
\end{cor}

\section{Hochschild homology and Hochschild classes}\label{Major}

\subsection{Hochshchild homology}\label{hochschild homology}

 In this subsection we recall the definition of Hochschild homology, (see \cite{Keller_loc}, \cite{lodcyclic_homology}) and prove that it can be expressed in terms of dualizing objects, (see \cite{CaldararuI}, \cite{caldararu}, \cite{KS3}, \cite{KoSY}).

\begin{defi}
The Hochshchild homology of a dg algebra is defined by 
\begin{equation*}
\Hoc(A):= \op{A} \Lte_{A^e} A.
\end{equation*}
The Hochschild homology groups are defined by $\HH_n(A)=\Hn^{-n}(\op{A} \Lte_{A^e} A)$.
\end{defi}

\begin{prop}\label{homisation}
If A is a proper and homologically smooth dg algebra then there is a natural isomorphism

\begin{equation} \label{dual description}
\Hoc(A) \simeq \RHom_{A^e}(\w_A^{-1},A).
\end{equation}
\end{prop}

\begin{proof}
We have 
\begin{align*}
\op{A} \Lte_{A^e} A & \simeq (\Du^{\prime}_{A^e} \circ \Du^{\prime}_{\eA} (\op{A})) \Lte_{A^e} A \\
               & \simeq   \Du^{\prime}_{A^e}(\w_A^{-1})\Lte_{A^e} A\\
               & \simeq  \RHom_{A^e}(\w_A^{-1},A).
\end{align*}%
\end{proof}
\begin{Rem}
There is also a natural isomorphism
\begin{equation*}
\Hoc(A) \simeq \RHom_{A^e}(A,\w_A).
\end{equation*}
It is obtain by adjunction from isomorphism (\ref{dual description}).
\end{Rem}

There is the following natural isomorphism.

\begin{prop}\label{split}
Let $A$ and $B$ be a dg algebras. Let $M \in \Der_{\mathrm{perf}}(A)$ and $S \in \Der_{\mathrm{perf}}(B)$ and $N \in \Der(A)$ and  $T \in \Der(B)$ then
\begin{equation*}
\RHom_{A \otimes B}(M \otimes S, N \otimes T) \simeq \RHom_{A}(M,N) \otimes \RHom_{B}(S,T).
\end{equation*}
\end{prop}

\begin{proof}
clear.
\end{proof}
A special case of the above proposition is
\begin{prop}[Künneth isomorphism] \label{Kunneth}
Let A and B be proper homologically smooth dg algebras. There is a natural isomorphism
\begin{equation*}
\Ku{A}{B}:\Hoc(A) \otimes \Hoc(B) \stackrel{\sim}{\rightarrow}\Hoc(A \otimes B).
\end{equation*}
\end{prop}

\subsection{The Hochschild class}\label{class}

 In this subsection, following \cite{KS3}, we construct the Hochschild class of an endomorphism of a perfect module and describe the Hochschild class of this endomorphism when the Hochschild homology is expressed in term of dualizing objects.\\
 
To build the Hochschild class, we need to construct some morphism of $\Der_{\mathrm{perf}}(A^e)$. 

\begin{lemme}\label{unit counit}
Let $M$ be a perfect $A$-module. There is a natural isomorphism

\begin{equation} \label{iso hhclass}
\xymatrix{
\RHom_A(M,M) \ar[r]^-{\sim}& \RHom_{A^e}(\omega^{-1}_A, M \otimes \Du^{\prime}_A M).}
\end{equation}
\end{lemme}

\begin{proof}
We have
\begin{equation*}
\begin{split}
\RHom_A(M,M) &\simeq \Du^{\prime}_A M \Lte_A M \\
             &\simeq \op{A} \Lte_{A^e} (M \otimes \Du^{\prime}_A M)\\
             &\simeq \RHom_{A^e}(\omega^{-1}_A, M \otimes \Du^{\prime}_A M).
\end{split}
\end{equation*}

Thus, we get an isomorphism

\begin{equation}
\xymatrix{
\RHom_A(M,M) \ar[r]^-{\sim}& \RHom_{A^e}(\omega^{-1}_A, M \otimes \Du^{\prime}_A M).}
\end{equation}
\end{proof}

\begin{defi} 
The morphism $\eta$  in $\Der_{\mathrm{perf}}(A^e)$ is the image of the identity of $M$ by morphism (\ref{iso hhclass}) and $\varepsilon$  in $\Der_{\mathrm{perf}}(A^e)$ is obtained from $\eta$ by duality.
\begin{equation} \label{preval}
\eta:\omega^{-1}_A \to M \otimes \Du^{\prime}_A M,\\
\end{equation}
\begin{equation} \label{eval}
\varepsilon:M \otimes \Du^{\prime}_A M \to A.
\end{equation}
The map $\eta$ is called the coevaluation map and $\varepsilon$ the evaluation map.
\end{defi}

 Applying $\Du^{\prime}_{A^e}$ to (\ref{preval}) we obtain a map 
\begin{equation*}
\Du^{\prime}_{A^e}( M \otimes \Du^{\prime}_A M) \to \op{A}.
\end{equation*}
Then using the isomorphism of Proposition \ref{split},                                      
\begin{equation*}
\begin{split}
\Du^{\prime}_{A^e}( M \otimes \Du^{\prime}_A M) & \simeq \Du^{\prime}_{A}( M)  \otimes \Du^{\prime}_{\op{A}}( \Du^{\prime}_A M)\\
                                                     & \simeq \Du^{\prime}_{A}( M)  \otimes M,
\end{split}
\end{equation*}
we get morphism (\ref{eval}).

Let us define the Hochschild class. We have  the following chain of morphisms
\begin{equation*}
\begin{split}
\RHom_A(M,M) & \simeq \Du^{\prime}_A M \Lte_A M\\
             & \simeq \op{A} \Lte_{A^e} ( M \otimes \Du^{\prime}_A M)\\
             & \stackrel{\id \otimes \varepsilon}{\longrightarrow} \op{A} \Lte_{A^e} A.
\end{split}
\end{equation*}%

We get a map 
\begin{equation} \label{image}
\Hom_{\Der_{\mathrm{perf}}(A)}(M,M) \rightarrow \HH_0(A).
\end{equation}

\begin{defi}
The image of an element $f \in \Hom_{\Dper(A)}(M,M)$  by the map (\ref{image}) is called the Hochschild class of $f$ and is denoted $\hh_A(M,f)$. The Hochschild class of the identity is denoted $\hh_A(M)$ and is called the Hochschild class of $M$.\\
\end{defi}
\begin{Rem}
If $A=k$ and $M \in \Dper(k)$, then  
\begin{equation*}
\hh_k(M,f)=\sum_i (-1)^i \tr(\Hn^i(f:M \rightarrow M)), 
\end{equation*}
see for instance \cite{KS3}.
\end{Rem}
\begin{lemme}\label{lemme de morphisation}
The isomorphism (\ref{dual description}) sends $\hh_A(M,f)$ to the image of $f$ under the composition
\begin{equation*}
\xymatrix{
\RHom_A(M,M) \ar[r]^-{\sim}& \RHom_{A^e}(\omega^{-1}_A, M \otimes \Du^{\prime}_A M) \ar[r]^-{\varepsilon \circ} & \RHom_{A^e}(\w_A^{-1},A)}
\end{equation*}
 where the first morphism is defined in (\ref{iso hhclass}) and the second morphism is induced by the evaluation map.
\end{lemme}

\begin{proof}
This follows from the commutative diagram:
$$
\xymatrix{ \RHom_A(M,M) \ar[d]_-{\id} \ar[r]^-{\sim}& \op{A} \Lte_{A^e} ( M \otimes \Du^{\prime}_A M) \ar[d] \ar[r]^-{\id_A \otimes \varepsilon}& \op{A} \Lte_{A^e} A              \ar[d]\\
            \RHom_A(M,M) \ar[r]^-{\sim} & \RHom_{A^e}(\w_A^{-1},M \otimes \Du^{\prime}_A M) \ar[r]^-{\varepsilon} & \RHom_{A^e}(\w_A^{-1},A)
.}
$$
\end{proof}  

%

\begin{Rem}
Our definition of the Hochschild class is equivalent to the definition of the trace of a 2-cell in \cite{ponto}.  
This equivalence allows us to use string diagrams to prove some properties of the Hochschild class, see \cite{ponto} and \cite{CaldararuI}, \cite{caldararu}. 
\end{Rem}

\begin{prop}\label{conjugaison}
Let $M, N \in \Der_{\mathrm{perf}}(A)$, $g \in \Hom_A(M,N)$ and $h \in \Hom_A(N,M)$ then
\begin{equation*}
\hh_A(N,g \circ h)= \hh_A(M, h \circ g).
\end{equation*}
\end{prop}

\begin{proof}
See for instance \cite[§7]{ponto}.
\end{proof}

\section{A pairing on Hochschild homology}
In this section, we  build a pairing on Hochschild homology. It acts as the Hochschild class of the diagonal, (see \cite{KS3}, \cite{CaldararuI}, \cite{caldararu}, \cite{shklyarov}). Using this result, we prove our Riemann-Roch type formula. We follow the approach of \cite{KS3}.

\subsection{Hochschild homology and bimodules}\label{spairing}

In this subsection, we translate to our language the classical fact that a perfect $A \otimes \op{B}$-module induces a morphism from $\Hoc(B)$ to $\Hoc(A)$. We need the following technical lemma which generalizes Lemma \ref{unit counit}. 

\begin{lemme} \label{technical}
Let $K \in \Der_{\mathrm{perf}}(A \otimes \op{B})$. Let $C=A \otimes \op{B}$. Then, there are natural morphisms in $\Der_{\mathrm{perf}}(A^e)$ which coincide with (\ref{preval}) and (\ref{eval}) when $B=k$,
\begin{equation*}
\omega^{-1}_A \rightarrow K \Lte_{B} \Du^{\prime}_{C} K,   
\end{equation*} 
\begin{equation} \label{evaluation gene}
K \Lte_{B} \w_B \Lte_{B} \Du^{\prime}_{C} K \rightarrow A. 
\end{equation}
\end{lemme}  
\begin{proof}
By (\ref{preval}), we have  a morphism in $\Der_{\mathrm{perf}}(A \otimes \op{B} \otimes B \otimes \op{A})$
\begin{equation*}
\w_{C}^{-1} \rightarrow K \otimes \Du^{\prime}_C K.
\end{equation*}
Applying the functor $-\Lte_{B^e} B$, we obtain
\begin{equation*}
\w_{C}^{-1} \Lte_{B^e} B \rightarrow K \otimes \Du^{\prime}_C K \Lte_{B^e} B
\end{equation*}

and by Proposition \ref{split} 
\begin{equation*}
\w_C^{-1}\simeq\w_A^{-1} \otimes \w_{\op{B}}^{-1} \simeq \w_A^{-1} \otimes \Du^{\prime}_{B^e}(B). 
\end{equation*}
Then there is a sequence of isomorphisms 
\begin{equation*}
\w_A^{-1} \otimes \RHom_{B^e}(B,B) \stackrel{\sim}{\to} \w_A^{-1} \otimes (\Du^{\prime}_{B^e} B \Lte_{B^e} B) \stackrel{\sim}{\to} \w_{C}^{-1} \Lte_{B^e} B
\end{equation*}
and there is a natural arrow $\w_A^{-1} \stackrel{\id \otimes 1}{\longrightarrow} \w_A^{-1} \otimes \RHom_{B^e}(B,B)$. Composing these maps, we obtain the morphism

\begin{equation*}
\w_A^{-1} \rightarrow \w_A^{-1} \otimes (\Du^{\prime}_{B^e} B \Lte_{B^e}  B) \rightarrow \w_{C}^{-1} \Lte_{B^e} B \rightarrow (K \otimes \Du^{\prime}_C K) \Lte_{B^e} B.
\end{equation*}

For the map (\ref{evaluation gene}), we have a morphism in $\Der_{\mathrm{perf}}(A \otimes \op{B} \otimes B \otimes \op{A})$ given by the map (\ref{eval})

\begin{equation*}
 K \otimes \Du^{\prime}_C K \rightarrow C.
\end{equation*} 
Then applying the functor $-\Lte_{B^e} \w_B$, we obtain
\begin{equation*}
(K \otimes \Du^{\prime}_C K) \Lte_{B^e} \w_B \rightarrow C \Lte_{B^e} \w_B \simeq (A \otimes \op{B})  \Lte_{B^e} \w_B.
\end{equation*}
Composing with the natural "integration" morphism of Proposition \ref{integration}, we get
\begin{equation*}
(K \otimes \Du^{\prime}_C K) \Lte_{B^e} \w_B \rightarrow A
\end{equation*}%
which proves the lemma.
\end{proof}

We shall show that an object in $\Der_{\mathrm{perf}}(A \otimes \op{B})$ induces a morphism between the Hochschild homology of $A$ and that of $B$.

Let $K \in \Der_{\mathrm{perf}}(A\otimes \op{B})$. We set $C=A \otimes \op{B}$ and $S=K \otimes (\w_B \Lte_B \Du^\prime_C(K)) \in \Der(A^e \otimes {\op{({B}^e)}})$.

We have
\begin{align*}
S \Lte_{B^e} \w_B^{-1} &\simeq K \Lte_B \w_B^{-1} \Lte_B \w_B \Lte_B \Du_C(K) && S \Lte_{B^e} B \simeq K \Lte_{B} \w_B \Lte_B \Du_C^\prime(K).\\
&\simeq K \Lte_B \Du^\prime_C(K) && 
\end{align*}
  
The map
\begin{equation} \label{phikconstruct} 
\Phi_K : \Hoc(B) \rightarrow \Hoc(A).
\end{equation}
is defined as follow.
\begin{align*}
\RHom_{B^e}(\w_B^{-1},B)& \to  \RHom_{C}(S \Lte_B \w_B^{-1} , S \Lte_B B)\\ 
&\to  \RHom_{A^e}(\w^{-1}_A,A). 
\end{align*}
The last arrow is associated with the morphisms in Lemma \ref{technical}. This defines the map (\ref{phikconstruct}).

\subsection{A pairing on Hochschild homology}\label{action pairing}

In this subsection, we build a pairing on the Hochschild homology of a dg algebra. It acts as the Hochschild class of the diagonal, (see \cite{KS3}, \cite{CaldararuI}, \cite{caldararu}, \cite{shklyarov}). We also relate $\Phi_K$ to this pairing.\\

A natural construction to obtain a pairing on the Hochshild homology of a dg algebra $A$ is the following one.

Consider $A$ as a perfect $k- {}^eA$ bimodule. The morphism (\ref{phikconstruct}) with $K=A$ provides a map 
\begin{equation*}
\Phi_A: \Hoc(\eA) \rightarrow \Hoc(k).
\end{equation*}
We compose $\Phi_A$ with $\Ku{\op{A}}{A}$ and get
\begin{equation}\label{hom pairing}
\Hoc(\op{A}) \otimes \Hoc(A) \rightarrow k.
\end{equation}
Taking the $0^th$ degree homology, we obtain
\begin{equation}\label{pairing crochet}
\langle \cdot , \cdot \rangle : (\bigoplus_{n \in \Z} \HH_{-n}(\op{A}) \otimes \HH_{n}(A)) \rightarrow k.
\end{equation}
 In other words $\langle \cdot , \cdot \rangle =\Hn^0( \Phi_A) \circ \Hn^{0}(\Ku{\op{A}}{A}).$
 
However, it is not clear how to express $\Phi_K$ in term of the Hochschild class of $K$ using the above construction of the pairing. Thus, we propose another construction of the pairing and shows it coincides with the previous one.

\begin{prop}\label{action}
Let $A, B, C$ be three proper homologically smooth dg algebras and $K$ an object of $\Dper(A \otimes \op{B})$.\\
There is a natural map
\begin{equation*}
\Hoc(A \otimes \op{B}) \otimes \Hoc(B \otimes \op{C}) \rightarrow \Hoc(A \otimes \op{C})
\end{equation*}
inducing, for every $i \in \mathbb{Z}$, an operation
\begin{equation*}
\cup_{B} : \bigoplus_{n \in \Z} (\HH_{-n}(A \otimes \op{B}) \otimes \HH_{n+i}(B \otimes \op{C}) )\rightarrow \HH_i(A \otimes \op{C}) 
\end{equation*}
such that for every $\lambda \in \HH_i(B \otimes \op{C})$, $\Hn^i(\Phi_{K}\otimes \id)(\lambda)=\hh_{A \otimes \op{B}}(K) \cup_{B}\lambda$.
\end{prop}

Before proving Proposition \ref{action}, let us do the following remark.

\begin{Rem}
Let $M \in \Der_{\mathrm{perf}}(A)$. There is an isomorphism in $\Der_{\mathrm{perf}}(k)$
\begin{equation*}
\w_A \Lte_{A} M \simeq M \Lte_{\op{A}} \w_{\op{A}}.
\end{equation*}
\end{Rem}

The next proof explains the construction of $\cup_B : \bigoplus_{n \in \Z} (\HH_{-n}(A \otimes \op{B}) \otimes \HH_{n+i}(B \otimes \op{C})) \rightarrow \HH_i(A \otimes \op{C})$. We also prove the equality $\Hn^i(\Phi_K \otimes \id)(\lambda)= \hh_{A\otimes \op{B}}(K) \cup_B \lambda$.

\begin{proof}[Proof of Proposition \ref{action} ]
\begin{enumerate}[(i)]
\item We identify $\op{(A \otimes \op{B})}$ and $\op{A} \otimes B$.\\ 
We have
\begin{equation*}
\begin{split}
\Hoc(A \otimes \op{B}) \simeq& \RHom_{A^e \otimes \eB}(\omega_{A \otimes \op{B}}^{-1},A \otimes \op{B})\\
\simeq& \RHom_{A^e \otimes \eB}(\omega_{A}^{-1} \otimes \omega_{\op{B}}^{-1} \Lte_{\op{B}} \omega_{\op{B}},A \otimes \op{B} \Lte_{\op{B}} \omega_{\op{B}})\\
 \simeq& \RHom_{A^e \otimes \eB}(\omega_{A}^{-1} \otimes \op{B},A \otimes \omega_{\op{B}}).
\end{split}
\end{equation*}

Let $S_{AB}=\omega_{A}^{-1} \otimes \op{B}$ and $T_{AB}=A \otimes \omega_{\op{B}}$. Similarly, we define $S_{BC}$ and $T_{BC}$. Then, we get\\ 
\begin{equation*}
\begin{split}
\Hoc(A \otimes \op{B}) \otimes & \Hoc(B \otimes \op{C}) \\
 &\simeq \RHom_{A^e \otimes \eB}(S_{AB},T_{AB}) \otimes  \RHom_{B^e \otimes \eC}(S_{BC},T_{BC})\\ 
                            & \to \RHom_{A^e \otimes \eC}(S_{AB} \Lte_{B^e} S_{BC}, T_{AB} \Lte_{B^e} T_{BC}).
\end{split}
\end{equation*}    
Using the morphism $k \to \RHom_{{}^e \!B}(\op{B},\op{B})$, we get
\begin{equation*}
 k \rightarrow \op{B} \Lte_{B^e} \w_{B}^{-1}.
\end{equation*}
Thus, we get

\begin{equation*}
\omega_{A}^{-1} \otimes \op{C} \rightarrow (\omega_{A}^{-1} \otimes \op{B}) \Lte_{B^e} (\w_{B}^{-1} \otimes \op{C}).
\end{equation*}
We know by Proposition \ref{integration} that there is a morphism
\begin{equation*}
\w_{\op{B}} \Lte_{B^e} B \rightarrow k.
\end{equation*}
we deduce a morphism
\begin{equation*}
(A \otimes \w_{\op{B}}) \Lte_{B^e}( B \otimes \omega_{\op{C}}) \rightarrow A \otimes \omega_{\op{C}}.
\end{equation*}%
Therefore we get
 \begin{equation*}
\Hoc(A \otimes \op{B}) \otimes \Hoc(B \otimes \op{C}) \rightarrow \Hoc(A \otimes \op{C}).
\end{equation*}%
Finally, taking the cohomology we obtain,
\begin{equation*}
\bigoplus_{n \in \Z}(\HH_{-n}(A \otimes \op{B}) \otimes \HH_{n+i}(B \otimes \op{C})) \rightarrow \HH_i(A \otimes \op{C}).
\end{equation*}

\item 
We follow the proof of \cite{KS3}. We only need to prove the case where $C=k$. The general case being a consequences of Lemma \ref{equivalence} (i) below. 
We set $P= A \otimes \op{B}$.
 Let $\alpha=\hh_{A \otimes \op{B}}(K)$. We assume that $\lambda \in \HH_0(B)$. The proof being similar if $\lambda \in \HH_i(B)$. By Proposition \ref{lemme de morphisation}, $\alpha$ can be viewed as a morphism of the form
\begin{equation*}
 \w_{A \otimes \op{B}}^{-1} \rightarrow K \otimes \Du^\prime_{P}(K) \rightarrow A \otimes \op{B}.
\end{equation*}
 
We consider $\lambda$ as a morphism $\w_{B}^{-1} \rightarrow B$. Then, following the construction of $\Phi_K$, we observe that $\Phi_{K}(\lambda)$ is obtained as the composition  
\begin{equation*}
\xymatrix{
\omega_{A}^{-1} \ar[r] & K \Lte_{B} \Du^{\prime}_{P} K \ar[r]^-{\lambda} & K \Lte_{B} \w_{B} \Lte_{B} \Du^{\prime}_{P} K \ar[r]& A
}
\end{equation*}%

We have the following commutative diagram in $\Der_{\mathrm{perf}}(k)$.
\begin{equation*}
\scalebox{0.77}{
\xymatrix @C=0.5cm @R=0.7cm{
\w_A^{-1} \ar[d]                                                      &                                                   &                   &      \\
(\w_A^{-1} \otimes \op{B}) \Lte_{B^e} \w_{B}^{-1}  \ar[r]^-{\lambda}\ar[d]^-{\wr}       & \w_A \otimes \op{B} \Lte_{B^e} B  \ar[d]^-{\wr}          &                   &       \\
((\w_{A}^{-1} \otimes \w_{\op{B}}^{-1}) \Lte_{\op{B}} \w_{\op{B}}) \Lte_{B^e} \w_{B}^{-1}    \ar[r]^-{\lambda} \ar[d]^-{\wr} &  ((\w_{A}^{-1} \otimes \w_{\op{B}}^{-1}) \Lte_{\op{B}} \w_{\op{B}}) \Lte_{B^e} B         \ar[d] \ar[rd]&         &      \\
(\w_{A \otimes B}^{-1} \Lte_{\op{B}} \w_{\op{B}}) \Lte_{B^e} \w_{B}^{-1} \ar[r]^-{\lambda} \ar[d]&   (\w_{A \otimes \op{B}}^{-1} \Lte_{\op{B}} \w_{\op{B}}) \Lte_{B^e} B \ar[d]          & ((K \otimes \Du^{\prime}_{P}K) \Lte_{\op{B}} \w_{\op{B}}) \Lte_{B^e} B \ar[ld]_-{\sim} \ar[d]&      \\
(K \otimes \w_B \Lte_{B} \Du^{\prime}_{P} K) \Lte_{B^e} \omega_{B}^{-1}\ar[r]^-{\lambda} \ar[d]^-{\wr}&  (K \otimes \w_B \Lte_{B} \Du^{\prime}_{P} K) \Lte_{B^e} B \ar[d]^-{\wr}          & (A \otimes \op{B} \Lte_{\op{B}} \w_{\op{B}}) \Lte_{B^e} B\ar[d]^-{\wr}       &      \\
K \Lte_{B} \w^{-1}_B \Lte_{B} \w_B \Lte_{B} \Du^{\prime}_P K   \ar[r]^-{\lambda}       & K \Lte_{B} B \Lte_{B} \w_B \Lte_{B} \Du^{\prime}_P K   \ar[d] & A\otimes \w_{\op{B}} \Lte_{B^e} B\ar[d]             &    \\
             & K \Lte_{B} \w_{B} \Lte_{B} \Du^{\prime}_P K \ar[r]          &      A.  & 
}%
}%
\end{equation*}

This diagram is obtained by computing $\Hn^0(\Phi_K)(\lambda)$ and $\alpha \cup \lambda$. The left column and the row on the bottom induces $\Hn^0(\Phi_K)(\lambda)$ whereas the row on the top and the right column induces $\alpha \cup \lambda$. This diagram commutes, consequently
$\Hn^0(\Phi_K)(\lambda)=\hh_{A \otimes \op{B}}(K) \cup \lambda$.
\end{enumerate}%
\end{proof}

We now give some properties of this operation.

\begin{lemme}\label{equivalence}
Let $A$, $B$, $C$ and $S$ be proper homologically smooth dg algebras, $\lambda_{A\op{B}} \in \HH_i(A \otimes \op{B})$. Then
\begin{enumerate}[(i)]

\item $\cup_B \circ (\cup_A \otimes \id)= \cup_A \circ (\id \otimes \cup_B)= \cup_{A \otimes \op{B}}$.


\item $\hh_{A \otimes \op{A}}(A) \cup_{A} \lambda_{A\op{B}}= \lambda_{A\op{B}}$  and  $\lambda_{A\op{B}} \cup_B \hh_{B \otimes \op{B}}(B)=\lambda_{A\op{B}}$. 
\end{enumerate}
\end{lemme}

\begin{proof}
\begin{enumerate}[(i)]
\item is obtained by a direct computation using the definition of $\cup$
\item results from Proposition \ref{action} \textit{(ii)} by noticing that $\Phi_A$ and $\Phi_B$ are equal to the identity.
\end{enumerate}
\end{proof}

From this natural operation we are able to deduce a pairing on Hochschild homology. Indeed using Proposition \ref{action} we obtain a pairing

\begin{equation}\label{cuppairing}
\cup : \bigoplus_{n \in \Z} (\HH_{-n}(\op{A}) \otimes \HH_n(A)) \rightarrow \HH_0(k)\simeq k. 
\end{equation} 
To relate the two preceding constructions of the pairing, we introduce a third way to construct it. Proposition $\ref{action}$ gives us a map
\begin{equation*}
\cup_{\eA} : \bigoplus_{n \in \Z}(\HH_{-n}(A^e) \otimes \HH_n(\eA)) \rightarrow \HH_0(k)\simeq k. 
\end{equation*} 
Then there is a morphism
\begin{align*}
\HH_{-n}(\op{A}) &\otimes \HH_n(A) \rightarrow k\\
\lambda \otimes \mu & \mapsto  \hh_{A^e}(A) \cup_{\eA} (\lambda \otimes \mu).
\end{align*}
Using Proposition \ref{action}, we get that 
\begin{equation*}
\Hn^i(\Phi_{A})(\lambda \otimes \mu)=\hh_{A^e}(A) \cup_{\eA} (\lambda \otimes \mu).
\end{equation*}
By Lemma \ref{equivalence}, we have 
\begin{align*}
\hh_{A^e}(A) \cup_{\eA} (\lambda \otimes \mu)&=(\lambda \cup_A \hh_{A \otimes \op{A}}(A)) \cup_{A} \mu\\
                                                                             &=\lambda \cup \mu.
\end{align*}


This proves that these three ways of defining a pairing lead to the same pairing. It also shows that the pairing is equivalent to the action of the Hochschild class of the diagonal.

\subsection{Riemann-Roch formula for dg algebras}\label{RR}

In this section we prove the Riemann-Roch formula announced in the introduction.

\begin{prop}\label{adapt}
Let $M \in \Der_{\mathrm{perf}}(\eA)$ and let $f \in \Hom_A(M,M)$. Then
\begin{equation*}
\hh_k(A \Lte_{\eA} M , \id_A \Lte_{\eA} f)= \hh_{A^e}(A) \cup \hh_{\eA}(M,f) 
\end{equation*}
\end{prop}

\begin{proof}
Let $\lambda=\hh_{\eA}(M,f) \in \HH_0(\eA)\simeq \Hom_{(\eA)^e}(\w_{\eA}^{-1},\eA)$. As previously, we set $B=\eA$. We denote by $\tilde{f}$ the image of $f$ in $\Hom_{B^e}(\w_B^{-1},M \otimes \Du^{\prime}_B M)$ by the isomorphism (\ref{iso hhclass}) and by $\widetilde{\id_A \Lte_{\eA}f}$ the image of $\id_A \Lte_{\eA}f$ by the isomorphism (\ref{iso hhclass}) applied with $A=k$ and $M= A \Lte_B M$. We obtain the commutative diagram below.

\begin{equation*}
\scalebox{0.88}{
\xymatrix{
\w_k^{-1} \ar[r] \ar@/_4pc/[ddr]_-{\widetilde{\id_A \Lte_{\eA}f}}& A \Lte_{B} \w_B^{-1} \Lte_{B} \w_B \Lte_B \Du^{\prime}_{\op{B}} A  \ar[r]^-{\lambda} \ar[d]^-{\tilde{f}}& A \Lte_{B} B \Lte_{B} \w_B \Lte_B \Du^{\prime}_{\op{B}} A \ar[r]& k\\
                & A \Lte_B (M \otimes \Du^{\prime}_{B} M) \Lte_B \w_B \Lte_B \Du^{\prime}_{\op{B}} A \ar[d]^-{\wr} \ar[ur]^-{\varepsilon}&& \\
                 & (A \Lte_B M) \otimes \Du^{\prime}_k(A \Lte_B M) \ar@/_3pc/[rruu]&&\\
}
}
\end{equation*}

The map $\w_k^{-1} \to A \Lte_{B} \w_B^{-1} \Lte_{B} \w_B \Lte_B \Du^{\prime}_{\op{B}} A$ is obtained by applying Lemma \ref{technical} with $A=k$, $B=\eA$, $K=A$. Then,
\begin{equation*}
\w_k^{-1} \to A \Lte_{B} \Du^{\prime}_{\op{B}} A \simeq A \Lte_{B} \w_B^{-1} \Lte_{B} \w_B \Lte_B \Du^{\prime}_{\op{B}} A.
\end{equation*}

The morphism $A \Lte_{B} B \Lte_{B} \w_B \Lte_B \Du^{\prime}_{\op{B}} A \to k$ is obtained as the composition of
\begin{equation*}
A \Lte_{B} B \Lte_{B} \w_B \Lte_B \Du^{\prime}_{\op{B}} A \simeq A \Lte_{B} \w_B \Lte_B \Du^{\prime}_{\op{B}} A
\end{equation*}
with
\begin{equation} \label{abovev}
A \Lte_{B} \w_B \Lte_B \Du^{\prime}_{\op{B}} A \to k.
\end{equation}
The morphism (\ref{abovev}) is the map (\ref{evaluation gene}) with $A=k$, $B=\eA$, $K=A$.

The vertical isomorphism is obtained by applying Theorem \ref{contraction} with $A=C=k$, $B=\op{B}$, $M=M$ and $N=A$.

By Lemma \ref{lemme de morphisation}, the composition of the arrows on the bottom is $\hh_k(A \Lte_{\eA} M, \id_A \Lte_{\eA} f)$ and the composition of the arrow on the top is $\Hn^0(\Phi_A(\hh_{\eA}(M,f)))$. It results from the commutativity of the diagram that
\begin{equation*}
\hh_k(A \Lte_{\eA} M , \id_A \Lte_{\eA} f)= \Hn^{0}(\Phi_A)(\hh_{\eA}(M,f)).
\end{equation*}
Then using Proposition \ref{action} we get
\begin{equation*}
\hh_k(A \Lte_{\eA} M , \id_A \Lte_{\eA} f)= \hh_{A^e}(A) \cup \hh_{\eA}(M,f). 
\end{equation*}
\end{proof}

We state and prove our main result which can be viewed as a noncommutative generalization of A.
C\u{a}ld\u{a}raru's version of the topological Cardy condition \cite{CaldararuI}.

\begin{thm}\label{main}
Let $M \in \Der_{\mathrm{perf}}(A), \; f \in \Hom_A(M,M)$ and $ N \in \Der_{\mathrm{perf}}(\op{A}), \; g \in \Hom_{\op{A}}(N,N)$.
Then
\begin{equation*}
\hh_k(N \Lte_{A} M, g \Lte_{A} f)=  \hh_{\op{A}}(N,g) \cup \hh_A(M,f).
\end{equation*}
where $\cup$ is the pairing defined by formula (\ref{cuppairing}).
\end{thm}

\begin{proof}
Let $u$ be the canonical isomorphism from $A \Lte_{\eA} (N \otimes M)$ to $ N \Lte_{A} M$. By definition of the pairing we have
\begin{equation*}
\begin{split}
\langle \hh_{\op{A}}(N,g) , \hh_A(M,f)\rangle &=\Hn^0(\Phi_A) \circ \Hn^{0}(\Ku{\op{A}}{A})(\hh_{\op{A}}(N,g) \otimes \hh_A(M,f))\\
                                       &= \hh_{A^e}(A) \cup \hh_{\eA}(N \otimes M , g \otimes f)\\
                                       &=\hh_k(A \Lte_{\eA} (N \otimes M), \id_A \Lte_{\eA} (g \otimes f))\\
                                       &=\hh_k(A \Lte_{\eA} (N \otimes M), u^{-1} \circ (g \Lte_A f) \circ u)\\
                                       &=\hh_k(N \Lte_{A} M, g \Lte_A f).
\end{split}
\end{equation*}
The last equality is a consequence of Proposition \ref{conjugaison}.
\end{proof}

\begin{Rem}
By adapting the proof of Proposition \ref{adapt}, we are also able to obtain the following result that should be compared to \cite[Theorem 4.3.4]{KS3}
\begin{thm}
Let $A$, $B$, $C$ be proper homologically smooth dg algebras. Let $K_1 \in \Der_{\mathrm{perf}}(A \otimes \op{B})$, $K_2 \in \Der_{\mathrm{perf}}(B \otimes \op{C})$, $f_1 \in \Hom_{A \otimes \op{B}}(K_1,K_1)$ and $f_2 \in \Hom_{B \otimes \op{C}}(K_2,K_2)$. Then
\begin{equation*}
\hh_{A \otimes \op{C}}(K_1 \Lte_B K_2, f_1 \Lte_B f_2)=\hh_{A \otimes \op{B}}(K_1,f_1) \cup_B \hh_{B \otimes \op{C}}(K_2,f_2). 
\end{equation*}
\end{thm} 
\end{Rem}

\bibliographystyle{smfplain}
\bibliography{bibliothese}

\end{document}